\documentclass{amsart}
\usepackage{graphicx}

\usepackage[utf8]{inputenc}
\usepackage{amsmath}
\usepackage{amsthm}
\usepackage{amsfonts}
\usepackage{amssymb}
\usepackage{tikz-cd} 
\usepackage{xcolor}
\usepackage{hyperref}
\usepackage{thmtools}
\usepackage{thm-restate}
\usepackage{comment}
\usepackage{caption}

\usepackage{mathtools}

\newcommand{\N}{\mathbb{N}}

\newcommand{\Z}{\mathbb{Z}}
\def\restrict#1{\raise-.5ex\hbox{\ensuremath|}_{#1}}

\theoremstyle{plain}
\newtheorem{theorem}{Theorem}[section]
\newtheorem{proposition}[theorem]{Proposition}

\newtheorem{corollary}[theorem]{Corollary}
\newtheorem{lemma}[theorem]{Lemma}

\theoremstyle{definition}
\newtheorem{definition}[theorem]{Definition}
\newtheorem{remark}[theorem]{Remark}
\numberwithin{equation}{section}
\newtheorem{example}[theorem]{Example}

\newtheorem{theoremx}{Theorem}

\newtheorem{conjecturex}[theoremx]{Conjecture}

\begin{document}

\markboth{Elder, Piggott and Townsend}{On $k$-geodetic graphs and groups}

\title{On $k$-geodetic graphs and groups}

\author{Murray Elder}
\address{University of Technology Sydney, Ultimo NSW 2007, Australia}
\email{murray.elder@uts.edu.au}

\author{Adam Piggott}
\address{Australian National University, Canberra ACT  2601, Australia}
\email{adam.piggott@anu.edu.au}

\author{Kane Townsend}
\address{University of Technology Sydney, Ultimo NSW 2007, Australia}
\email{kane.townsend@uts.edu.au}

\makeatletter
\@namedef{subjclassname@2020}{\textup{2020} Mathematics Subject Classification}
\makeatother



\maketitle
\vspace{-2em}
\begin{abstract}
We call a graph $k$-geodetic, for some $k\geq 1$,  if it is connected and between any two vertices there are at most $k$ geodesics. It is shown that any hyperbolic group with a $k$-geodetic Cayley graph is virtually-free. Furthermore, in such a group the centraliser of any infinite order element is an infinite cyclic group. These results were  known previously only in the case that $k=1$. A key tool used to develop the theorem is a new graph theoretic result concerning ``ladder-like structures'' in a $k$-geodetic graph.
\end{abstract}

\section{Introduction}

For any positive integer $k$, we will call a (possibly infinite) graph \emph{$k$-geodetic} if the graph is connected and between any pair of vertices there are at most $k$ geodesics. For example, a tree is $1$-geodetic and the complete bipartite graph $K_{k,l}$ is $(\max\{k,l\})$-geodetic. While 1-geodetic graphs (known simply as \emph{geodetic} graphs) \cite{Ore, Frasser} and  2-geodetic graphs \cite{bigeodetic} have been studied, it seems that little work has been done on $k$-geodetic graphs. Our first result is a necessary condition for a graph to be $k$-geodetic. We introduce a technical notion of a \emph{ladder-like structure} with parameters for height and width (see Definition \ref{lls}).

\begin{theoremx}\label{LadderTheorem}
Let $m$ and $k$ be positive integers. In any $k$-geodetic graph there is a universal bound on the height of ladder-like structures of width $m$. 
\end{theoremx}

 A group $G$ is called \emph{$k$-geodetic} if it admits a finite inverse-closed generating set $S$ such that the corresponding undirected Cayley graph $\text{Cay}(G, S)$ is $k$-geodetic.  It is clear that any finite group $G$ is geodetic (with $S = G \setminus \{1_G\}$). The hyperbolic groups are a natural next class of groups to investigate.  If $G$ is hyperbolic, then geodesics fellow travel and we may use this property to construct ladder-like structures. We parlay this idea into our second result which demonstrates that the hyperbolic $k$-geodetic groups form a proper subclass of the virtually-free groups.  

\begin{theoremx}\label{TheTheorem}
Let $k$ be a positive integer.  If $G$ is a hyperbolic $k$-geodetic group, then $G$ is virtually-free and in $G$ the centraliser of any infinite order element is an infinite cyclic group.
\end{theoremx}

We note that $\mathbb{Z} \times \mathbb{Z}_2$ fails the centraliser condition of Theorem \ref{TheTheorem} and so is an example of a virtually-free group that is not $k$-geodetic for any positive integer $k$.

In 1997, Shapiro \cite{ShapiroPascal} asked if the geodetic groups are exactly the \emph{plain} groups. A group is \emph{plain} if it is isomorphic to a free product of finitely many finite groups and finitely many copies of $\mathbb{Z}$. There is a natural choice of generating set of a plain group so that the Cayley graph is geodetic. Although Shapiro's question remains unanswered in general, some progress has been made in the special case of hyperbolic groups. Papasoglu \cite[1.4]{Papasoglu} showed that hyperbolic geodetic groups are in fact virtually-free. Observing that hyperbolic geodetic groups admit presentation by finite Church-Rosser Thue systems, one may apply a result by Madlener and Otto\cite{MadlenerOtto} to conclude that in hyperbolic geodetic groups the centraliser of any infinite order element is infinite cyclic. Theorem \ref{TheTheorem} shows that hyperbolic $k$-geodetic groups satisfy the key constraints known to hold for hyperbolic geodetic groups.

Shapiro \cite[p.6]{ShapiroPascal} proved that if $G$ is virtually infinite cyclic and $k$-geodetic with respect to generating set $S$, then $G$ is isomorphic to either $\mathbb{Z}$ or $\mathbb{Z}_2 \ast \mathbb{Z}_2$ and $S$ is the standard generating set. Taken with the existing theory, Theorem \ref{TheTheorem} leaves us with the containments in Fig. \ref{fig:Containments} and is evidence in favour of the following conjecture. 
\begin{figure}[t]
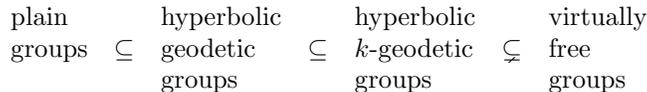

    \centering
\[\begin{array}{lllllll}
\text{plain} &  & \text{hyperbolic} &  & \text{hyperbolic} & & \text{virtually}\\
\text{groups} & \subseteq & \text{geodetic} & \subseteq & k\text{-geodetic} & \subsetneq & \text{free} \\
& & \text{groups} & & \text{groups} && \text{groups}
\end{array}\]
    \caption{\small Known containments.}
    \label{fig:Containments}
\end{figure}

\begin{conjecturex}\label{geodeticequalskgeodetic}
A hyperbolic group $G$ is $k$-geodetic for some positive integer $k$ if and only if $G$ is geodetic. Furthermore, if $G$ is infinite and $\text{Cay}(G,S)$ is $k$-geodetic for some finite generating set $S$ and some $k\geq 1$, then $\text{Cay}(G,S)$ is geodetic.
\end{conjecturex}

We note the difference between finite and infinite groups in the above conjecture. For any finite group $G$, it is clear that $\text{Cay}(G, G \setminus \{1_G\})$ is geodetic.
For any positive integer $k$, the only infinite $k$-geodetic Cayley graphs we know are in fact geodetic. 

\begin{example}\label{kgeodeticeg}
For any integer $k > 1$, we give an example of a group $G$ and generating set $S$ such that $\text{Cay}(G,S)$ is $k$-geodetic but not $(k-1)$-geodetic as follows. We observe that the complete bipartite graph $K_{k,k}$ is $k$-geodetic and not $(k-1)$-geodetic. We now choose a group and generating set with Cayley graph $K_{k,k}$. Let $G$ be the cyclic group of order $2k$, let $a$ be an order $2k$ element in $G$ and define $S:=\{a^{2i+1}\mid 0\leq i\leq k-1\}$. Then $\text{Cay}(G,S)$ has $k$ distinct geodesics of length $2$ for each $a^{2i}\in G$ with $1 \leq i\leq k$ and a unique geodesic of length $1$ for each $a^{2i+1}\in G$ with $0\leq i\leq k-1$.
\end{example}

\section{Preliminaries}

Let $X=(V,E)$ be a locally-finite simple connected graph. For $a,b\in\mathbb{N}$ with $a\leq b$, define $[a,b]$ to be $\{a,a+1,\dots,b\}$. A \emph{path} in $X$ is a map $\gamma :[0,n]\to V$ with $\{v_i,v_{i+1}\} \in E$ for each $0\leq i\leq n-1$. The path $\gamma$ has an \emph{initial point}, \emph{end point} and \emph{length} given by $\gamma(0)$, $\gamma(n)$ and $n$ respectively.

There is a metric $d_X\colon V\times V \to \mathbb{N}$ such that $d_X(u, v)$ is the length of a minimal length path between $u$ and $v$. We call such a path a \emph{geodesic}. We say that $X$ is \emph{$k$-geodetic} if for any pair of vertices the number of distinct geodesics between them is less than or equal to $k$. In the special case that $k = 1$, we say that $X$ is \emph{geodetic}. For our arguments that follow, we will require precise notions relating to fellow travelling.

\begin{definition}\label{fellowtravel}
Let $\gamma_i\colon[0,n_i]\to V$ for $i\in\{1,2\}$ be paths in $V$ and $n=\max\{n_1,n_2\}$. Then
the paths are said to \emph{$m$-fellow travel} if \[d_X(\gamma_1(t),\gamma_2(t))\in[0,m],\] for all $t\in [0,n]$. Note that if $n_i<n$, we define $\gamma_i(t)=\gamma_i(n_i)$ for all $t \in [n_i+1,n]$.
\end{definition}

\begin{definition}\label{allthedefinitions}
Let $\gamma_1,\gamma_2\colon[0,n]\to V$ be paths of length $n$. For a given $m>0$, we say $\gamma_1$ and $\gamma_2$ are:
\begin{itemize}
        \item[(i)] \emph{$m$-apart at $i\in[0,n]$} if $d_X(\gamma_1(i),\gamma_2(i))=m$; 
            \item[(ii)] \emph{$m$-close at $i\in[0,n]$} if $ d_X(\gamma_1(i),\gamma_2(i))\in[1,m]$;
    \item[(iii)]  \emph{asynchronously disjoint} if for all distinct $i,j\in[0,n]$ we have $\gamma_1(i) \neq \gamma_2(j)$;
    \item[(iv)] \emph{co-travelling} if $\gamma_{1}(i)=\gamma_{2}(j)$ and $\gamma_{1}(i+1)=\gamma_{2}(j+1)$ for some $i,j \in [0,n-1]$, and \emph{synchronously co-travelling} if $i=j$.
\end{itemize}
Furthermore, we define $a_m(\gamma_1,\gamma_2):=|\{i\in[0,n] \mid d_X(\gamma_1(i),\gamma_2(i))=m \}|$ and $c_m(\gamma_1,\gamma_2):=|\{i\in[0,n] \mid d_X(\gamma_1(i),\gamma_2(i))\in [1,m] \}|$; so $a_m(\gamma_1, \gamma_2)$ records the number of times that $\gamma_1$ and $\gamma_2$ are $m$-apart, while $c_m(\gamma_1, \gamma_1)$ records the number of times they are $m$-close.
\end{definition}

\begin{definition}
A \emph{geodesic triangle} in $X$ is the union of three geodesic paths $\alpha\colon [0,n_\alpha]\to V$, $\beta\colon [0,n_\beta]\to V$ and $\gamma\colon [0,n_\gamma]\to V$, such that $\alpha(n_\alpha)=\beta(0), \beta(n_\beta)=\gamma(0)$ and $\gamma(n_\gamma)=\alpha(0)$. The geodesic triangle is \emph{non-degenerate} if $\alpha(a),\beta(b),\gamma(c)$ are pairwise distinct for all $a\in[1,n_\alpha]$, $b\in[1,n_\beta]$ and $c\in[1,n_\gamma]$; otherwise it is \emph{degenerate}.
\end{definition}

\begin{definition}
A \emph{geodesic bigon} in $X$ is the union of two geodesic paths $\alpha\colon [0,n]\to V$ and $\beta\colon [0,n]\to V$ such that $\alpha(0)=\beta(0)$ and $\beta(n)=\gamma(n)$. The geodesic bigon is \emph{non-degenerate} if $\alpha(i)\neq \beta(i)$ for all $i\in[1,n-1]$; otherwise it is \emph{degenerate}.
\end{definition}

Let $G$ be a group and $S\subseteq G \setminus \{1_G\}$ a finite inverse-closed generating set. The undirected Cayley graph of $G$ with respect to $S$, denoted $\text{Cay}(G,S)$ is the graph with vertex set $G$ and edge set $\{\{g, h\} \in G \times G \mid g^{-1} h \in S\}$. Since $S$ generates $G$, $\text{Cay}(G,S)$ is connected. Since $S$ is finite, $\text{Cay}(G,S)$ is locally-finite. Since $1_G \not \in S$ and $S\subset G$, $\text{Cay}(G,S)$ is simple. We call $S$ an \emph{alphabet} and denote the set of finite words over the alphabet $S$ by $S^\ast$. We write $|u|$ for the length of the word $u\in S^\ast$; the unique word of length 0 is called the empty word and denoted $\lambda$. Let $S^+:=S^\ast\setminus\{\lambda\}$. For any $w=w_1w_2\dots w_n \in S^\ast$, a word of the form $w_iw_{i+1}\dots w_j$ with $1\leq i\leq j\leq n$ is called a \emph{factor} of $w$. A word $w\in S^+$ is called \emph{primitive} if there is no word $u\in S^\ast$ such that $w=u^m$ for some $m>1$. If a word $w$ is not primitive, then we call the minimal length word $u$ such that $w=u^m$ for some $m>1$ the \emph{primitive root} for $w$. For any $g\in G$, we write $|g|_{G,S}$ for the length of a shortest word $w \in S^\ast$ such that $w$ spells $g$. For every $u \in G$, there is a bijective correspondence between paths in $\text{Cay}(G,S)$ with initial vertex $u$ and words in $S^\ast$; minimal length words spelling a group element $g$ correspond to geodesic paths in $\text{Cay}(G,S)$ from $u$ to $ug$. We write $u=v$ if $u,v\in S^\ast$ are identical as words. We use use the symbol $\equiv$ to denote that the left hand side and right hand side evaluate to the same element in $G$. For any $g\in G$ and $r>0$, we write $B_{r}(g)$ for the set $\{h\in G \mid d_{X}(g,h)<r\}$. The centraliser of an element $g\in G$ is defined to be $C_G(g):=\{h\in G \mid gh\equiv hg \}$.

We refer the reader to \cite{MetricSpacesofNon-PositiveCurvature} for basic definitions and results regarding hyperbolic geodesic metric spaces. A locally-finite simple connected graph $X$ is a geodesic metric space. Let $T$ be a geodesic triangle in $X$ with vertices $T_1,T_2$ and $T_3$ and sides $\gamma_1,\gamma_2$ and $\gamma_3$.
\begin{figure}[h]
\centering
\begin{tikzpicture}
\draw (0, 0) to [bend right=20] (2, 2);
\draw (2, 2) to [bend right=20] (4, 0);
\draw (0,0) to [bend left=20] (4,0);
\node at (0,0)[circle,fill,inner sep=1pt]{};
  \node at (2,2)[circle,fill,inner sep=1pt]{};
  \node at (4,0)[circle,fill,inner sep=1pt]{};
\node at (2.63,1)[circle,fill,inner sep=1pt]{};
\node at (2,0.38)[circle,fill,inner sep=1pt]{};
\node at (1.37,1)[circle,fill,inner sep=1pt]{};
\node at (-0.2,0) {${T_1}$};
\node at (2,2.2) {${T_2}$};
\node at (4.2,0) {${T_3}$};
\node at (3,1) {${p_1}$};
\node at (2,0) {${p_2}$};
\node at (1,1) {${p_3}$};
\draw (2.63,1) to  (2,0.38);
\draw (2.63,1) to  (1.37,1);
\draw (1.37,1) to (2,0.38);
\end{tikzpicture}
\caption{\small Hyperbolic space has $\delta$-thin geodesic triangles}
\end{figure}
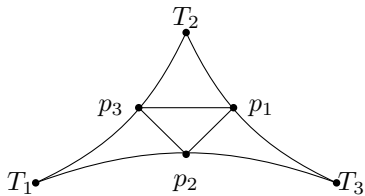
Take the points $p_i$ on each $\gamma_i$ to be those that have \[d_X(T_1,p_2)=d_X(T_1,p_3), d_X(T_2,p_1)=d_X(T_2,p_3), d_X(T_3,p_1)=d_X(T_3,p_2).\] For a real number $\delta>0$, we say $T$ is \emph{$\delta$-thin} if for each $i\in\{1,2,3\}$ and distinct $j,k\in\{1,2,3\}\setminus\{i\}$ the sub-paths of $\gamma_j$ and $\gamma_k$ from $T_i$ to $p_j$ and $p_k$ respectively $\delta$-fellow travel. We say $X$ is \emph{hyperbolic} if there exists $\delta>0$ such that all geodesic triangles are $\delta$-thin. We say a group $G$ is a \emph{hyperbolic group} if $\text{Cay}(G,S)$ is a hyperbolic for some (and hence any) finite generating set $S$. We have the well-known fellow traveller property in hyperbolic groups \cite[Lemma 2.3.2 and Thm. 3.4.5]{WordProcessingInGroups}:

\begin{proposition}\label{fellowtravellerproperty}
Let $G$ be a hyperbolic group and $X=\text{Cay}(G,S)$ for some finite generating set $S$. Then for any $c\geq 0$ there exists an $m_c>0$ such that any two geodesics $\gamma_i:[0,n_i]\to X$ with $i\in\{1,2\}$, $\gamma_1(0)=\gamma_2(0)$ and $d_X(\gamma_1(n_1),\gamma_2(n_2))\leq c$ will $m_c$-fellow travel.
\end{proposition}

\section{Ladder-like structures are bounded}

We will show that in a $k$-geodetic graph, there is a bound on the number of times a pair of asynchronously disjoint geodesics may be $m$-apart and $m$-close.

\begin{lemma}\label{OneOrTwoLess}
Let $X$ be a $k$-geodetic graph and let $u, v$ be vertices in $X$.  If there exist distinct paths $\alpha_0, \dots, \alpha_k\colon [0,n]\to X$ with initial point $u$ and terminal point $v$, then there exists a path $\beta$ from $u$ to $v$ of length $n-1$ or $n-2$.
\end{lemma}

\begin{proof}
Since there are $k+1$ paths of length $n$, none of them can be geodesics. Consider the sequence of paths $\alpha_0\restrict{[0,i]}$ for $i\in [0,n]$. Let \[i_0:=\min\{i\in [0,n]  \mid \alpha_0\restrict{[0,i]} \ \text{is not a geodesic}\}.\] Define $\beta_0$ to be a geodesic from $u$ to $\alpha_0(i_0)$. Then $\beta_0$ has length $j$ for some $j\in[i_0-2,i_0-1]$, since a shorter path contradicts the minimality of $i_0$. Define a path $\beta$ by \[\beta(i):=\begin{cases} 
      \beta_0(i) & \text{for} \ i\in[0,j], \\
      \alpha_0(i+i_0-j) & \text{for} \ i\in[j+1,n-i_0+j]. \\
   \end{cases}\] Then $\beta$ is a path from $u$ to $v$ with length $n-1$ or $n-2$.
\end{proof}

\begin{definition}\label{lls}
Let $m$ and $r$ be positive integers. A \emph{ladder-like structure} of \emph{width} $m$ and  \emph{height} $r$ is a pair of asynchronously disjoint geodesics $\gamma_x$ and $\gamma_y$ with $a_m(\gamma_x,\gamma_y)=r$.
\end{definition}

\begin{proposition}\label{kboundedmapart}
Let $m$ and $k$ be positive integers. There exists a constant $A(m, k)$ such that no ladder-like structure of width $m$ has a height exceeding $A(m,k)$ in any $k$-geodetic graph.
\end{proposition}

\begin{proof}
Let $k$ and $m$ be positive integers. Define $r:=k\prod_{i=2}^{2m+1}(ik+1)$ and $A(m,k):=mr$. Let $X$ be a $k$-geodetic graph. For contradiction, suppose there exist two asynchronously disjoint geodesics $\gamma_x$ and $\gamma_y$ in $X$ that form a ladder-like structure of width $m$ and height $A(m,k)+1$. For each $i\in[0,r]$, define the points $x_i$ on $\gamma_x$ and $y_i$ on $\gamma_y$ to be $(im+1)$-{th} occurrence of $\gamma_x$ and $\gamma_y$ being $m$-apart, ignoring all other occurrences that $\gamma_x$ and $\gamma_y$ are $m$-apart. Hence, there exists a diagram for $\gamma_x$ and $\gamma_y$ where each $d_i\geq m$ as depicted in Fig. \ref{fig:LL1}. The top row from $x_0$ to $x_r$ is a depiction of $\gamma_x$, the bottom row from $y_0$ to $y_r$ is a depiction of $\gamma_y$ and $d_i\geq m$ for each $i\in[1,r]$. For each $j\in[0,r]$, the path from $x_j$ to $y_j$ is a geodesic $\gamma_j$ of length $m$. The vertices in $\{x_0, \dots, x_r\} \cup \{y_0, \dots, y_r\}$ are pairwise disjoint: because $\gamma_x$ is a geodesic, $x_i = x_j$ if and only if $i = j$; because $\gamma_y$ is a geodesic, $y_i = y_j$ if and only of $i = j$; because $\gamma_x$ and $\gamma_y$ are asynchronously disjoint, $x_i \neq y_j$ for any $i, j$ such that $i \neq j$; because the ladder has width $m$, $d(x_i, y_i) = m > 0$ for any $i$. Furthermore, since $d_i\geq m$ for each $i\in[1,r]$ we have that $x_{i+1}$ does not lie on $\gamma_i$ for any $i$. For clarity in the arguments to follow, we schematically depict this part of the graph as shown in Fig. \ref{fig:LL2}.

\begin{figure}[ht]
    \centering
\begin{tikzpicture}
  \draw (1, 6.5) to[bend left] (3, 7);
  \draw (3, 7) to[bend right] (5, 6.5);
   \draw (5, 7) to (5.67, 6.25);
   \draw (5.67, 6.25) to (6.33, 6.25);
      \draw (6.33, 6.25) to (7,7);
\draw (7,7) sin (7.5,7.5) cos (8,7) sin (8.5,6.5) cos (9,7);
     \draw (5, 5.5) to (5.67, 6.25);
      \draw (6.33, 6.25) to (7,5.5);
 \draw (3,5.5) sin (3.5,6) cos (4,5.5) sin (4.5,5) cos (5,5.5); 
 \draw (7, 5.5) to (9, 5.5);

  \draw (1, 6) to (3, 5.5);
    \draw (10, 7) to[bend left] (12, 7);
  \draw (10, 5.5) to[bend right] (12, 5.5);
   \draw (9, 7)[dashed] to (10, 7);
  \draw (9, 5.5)[dashed] to (10, 5.5);
  \node at (1,7)[circle,fill,inner sep=1pt]{};
  \node at (1,5.5)[circle,fill,inner sep=1pt]{};
  \node at (3,7)[circle,fill,inner sep=1pt]{};
  \node at (3,5.5)[circle,fill,inner sep=1pt]{};
  \node at (5,7)[circle,fill,inner sep=1pt]{};
  \node at (5,5.5)[circle,fill,inner sep=1pt]{};
  \node at (7,7)[circle,fill,inner sep=1pt]{};
  \node at (7,5.5)[circle,fill,inner sep=1pt]{};
  \node at (9,7)[circle,fill,inner sep=1pt]{};
  \node at (9,5.5)[circle,fill,inner sep=1pt]{};
  \node at (10,7)[circle,fill,inner sep=1pt]{};
  \node at (10,5.5)[circle,fill,inner sep=1pt]{};
  \node at (12,7)[circle,fill,inner sep=1pt]{};
  \node at (12,5.5)[circle,fill,inner sep=1pt]{};
  \draw (1,5.5) to (1,7);
  \draw (3,5.5) to (3,7);
  \draw (5,5.5) to (5,7);
  \draw (7,5.5) to (7,7);
  \draw (9,5.5) to (9,7);
  \draw (10,5.5) to (10,7);
\draw (12,5.5) to (12,7);
\node at (1,7.2) {$x_0$};
\node at (3,7.2) {$x_1$};
\node at (5,7.2) {$x_2$};
\node at (6.9,7.2) {$x_3$};
\node at (9,7.2) {$x_4$};
\node at (9.9,7.2) {$x_{r-1}$};
\node at (12,7.2) {$x_r$};
\node at (1,5.3) {$y_0$};
\node at (3,5.3) {$y_1$};
\node at (5.1,5.3) {$y_2$};
\node at (7,5.3) {$y_3$};
\node at (9,5.3) {$y_4$};
\node at (10,5.3) {$y_{r-1}$};
\node at (12,5.3) {$y_r$};
\node at (2,7.3) {$d_1$};
\node at (4,6.7) {$d_2$};
\node at (6,6.7) {$d_3$};
\node at (8.1,7.3) {$d_4$};
\node at (11,7.1) {$d_{r}$};
\node at (2,5.5) {$d_1$};
\node at (4,5.2) {$d_2$};
\node at (6,5.9) {$d_3$};
\node at (8,5.3) {$d_4$};
\node at (11,5.5) {$d_{r}$};
\node at (0.8,6.25) {$m$};
\node at (2.8,6.25) {$m$};
\node at (4.8,6.25) {$m$};
\node at (6.8,6.25) {$m$};
\node at (8.8,6.25) {$m$};
\node at (9.8,6.25) {$m$};
\node at (11.8,6.25) {$m$};
\end{tikzpicture}
\caption{\small An example ladder-like structure of width $m$ and height $h$.}
    \label{fig:LL1}
\end{figure}
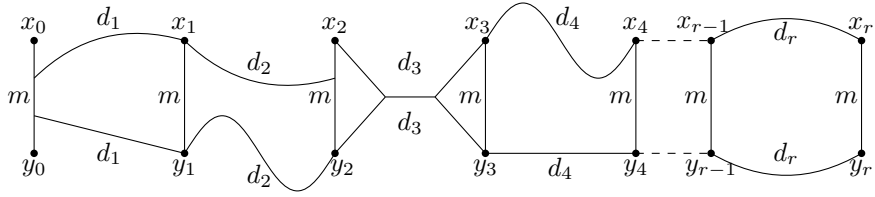

\begin{figure}[ht]
    \centering
\begin{tikzpicture}
  \draw (1, 7) to (9, 7);
  \draw (1, 5.5) to (9, 5.5);
    \draw (10, 7) to (12, 7);
  \draw (10, 5.5) to (12, 5.5);
   \draw (9, 7)[dashed] to (10, 7);
  \draw (9, 5.5)[dashed] to (10, 5.5);
  \node at (1,7)[circle,fill,inner sep=1pt]{};
  \node at (1,5.5)[circle,fill,inner sep=1pt]{};
  \node at (3,7)[circle,fill,inner sep=1pt]{};
  \node at (3,5.5)[circle,fill,inner sep=1pt]{};
  \node at (5,7)[circle,fill,inner sep=1pt]{};
  \node at (5,5.5)[circle,fill,inner sep=1pt]{};
  \node at (7,7)[circle,fill,inner sep=1pt]{};
  \node at (7,5.5)[circle,fill,inner sep=1pt]{};
  \node at (9,7)[circle,fill,inner sep=1pt]{};
  \node at (9,5.5)[circle,fill,inner sep=1pt]{};
  \node at (10,7)[circle,fill,inner sep=1pt]{};
  \node at (10,5.5)[circle,fill,inner sep=1pt]{};
  \node at (12,7)[circle,fill,inner sep=1pt]{};
  \node at (12,5.5)[circle,fill,inner sep=1pt]{};
  \draw (1,5.5) to (1,7);
  \draw (3,5.5) to (3,7);
  \draw (5,5.5) to (5,7);
  \draw (7,5.5) to (7,7);
  \draw (9,5.5) to (9,7);
  \draw (10,5.5) to (10,7);
\draw (12,5.5) to (12,7);
\node at (1,7.2) {$x_0$};
\node at (3,7.2) {$x_1$};
\node at (5,7.2) {$x_2$};
\node at (7,7.2) {$x_3$};
\node at (9,7.2) {$x_4$};
\node at (10,7.2) {$x_{r-1}$};
\node at (12,7.2) {$x_r$};
\node at (1,5.3) {$y_0$};
\node at (3,5.3) {$y_1$};
\node at (5,5.3) {$y_2$};
\node at (7,5.3) {$y_3$};
\node at (9,5.3) {$y_4$};
\node at (10,5.3) {$y_{r-1}$};
\node at (12,5.3) {$y_r$};
\node at (2,7.2) {$d_1$};
\node at (4,7.2) {$d_2$};
\node at (6,7.2) {$d_3$};
\node at (8,7.2) {$d_4$};
\node at (11,7.2) {$d_{r}$};
\node at (2,5.3) {$d_1$};
\node at (4,5.3) {$d_2$};
\node at (6,5.3) {$d_3$};
\node at (8,5.3) {$d_4$};
\node at (11,5.3) {$d_{r}$};
\node at (0.8,6.25) {$m$};
\node at (2.8,6.25) {$m$};
\node at (4.8,6.25) {$m$};
\node at (6.8,6.25) {$m$};
\node at (8.8,6.25) {$m$};
\node at (9.8,6.25) {$m$};
\node at (11.8,6.25) {$m$};
\end{tikzpicture}
\caption{\small A schematic ladder-like structure of width $m$ and height $h$}
    \label{fig:LL2}
\end{figure}
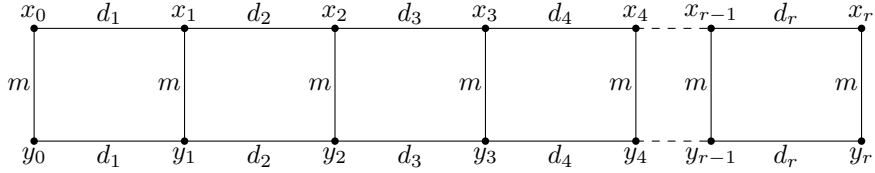

In this paragraph we demonstrate the existence of a `short' path from $x_0$ to $y_k$.  For each $j \in [0,k]$ we define a path $\alpha_j$ from $x_0$ to $y_k$ as follows: $\alpha_j$ travels along $\gamma_x$ from $x_0$ to $x_j$, then travels along $\gamma_j$ to $y_j$, and finally travels along $\gamma_y$ to $y_k$.  Although the paths $\alpha_0, \dots, \alpha_j$ are not necessarily internally disjoint, they are distinguished by which of the points $\{x_0, \dots, x_r\}\cup \{y_0, \dots, y_r\}$ they visit. Hence, we have exhibited $k+1$ distinct paths from $x_0$ to $y_k$ of length $m+\sum_{i=1}^{k}d_i$. By Lemma \ref{OneOrTwoLess}, there is a path $\beta$ from $x_0$ to $y_k$ such that \[|\beta| = m-1+\sum_{i=1}^{k}d_i \ \text{or} \ |\beta| = m-2+\sum_{i=1}^{k}d_i.\] Since $\gamma_x$ and $\gamma_k$ are geodesics, any path from $x_0$ to $y_k$ that passes through $x_k$ has length at least $m + \sum_{i=1}^{k}d_i$; hence $\beta$ does not pass through $x_k$. 

For each $j\in[1,2k]$, we repeat the above argument for paths from $x_{jk}$ to $y_{(j+1)k}$. We deduce that for each $j\in[0,2k]$, there is a path from $x_{jk}$ to $y_{(j+1)k}$ that does not pass through $x_{(j+1)k}$ and has length \[m-1+\sum_{i=jk+1}^{(j+1)k}d_i \ \text{or} \ m-2+\sum_{i=jk+1}^{(j+1)k}d_i.\] For each $j\in[0,2k]$, extend these paths so that their initial vertex is $x_0$, travelling along $\gamma_x$ to $x_{j}$, and their terminal vertex is $y_{k(2k+1)}$, travelling along $\gamma_y$ from $y_{(j+1)k}$. By the pigeonhole principle, at least $k+1$ of the extended paths have the same length. Since $X$ is $k$-geodetic, by Lemma \ref{OneOrTwoLess} there is a path from $x_0$ to $y_{k(2k+1)}$ of length \[m-p+\sum_{i=1}^{k(2k+1)}d_i \] with $p\in[2,4]$. Continuing these arguments we see that there is a path of length \[m-p+\sum_{i=1}^{f(c)}d_i\]
from $x_0$ to $y_{f(c)}$ for some $p\in[c,2c]$, where $f(c)\:=k\prod_{i=2}^{c}(ik+1)$. By our assumption we can take $c=2m+1$, which gives $p>2m$. This implies the existence of a path from $x_0$ to $x_r$ going via $y_r$ that is shorter than travelling along the geodesic $\gamma_x$. We have a contradiction showing the ladder-like structure with width $m$ cannot have height exceeding $A(m,k)$.
\end{proof}

This completes the proof of Theorem \ref{LadderTheorem}.

\begin{corollary}\label{kboundedmclose}
Let $m$ and $k$ be positive integers. There exists a constant $C(m,k)$ such that asynchronously disjoint geodesics cannot be $m$-close more than $C(m,k)$ times in any $k$-geodetic graph.
\end{corollary}
 
\begin{proof}
The result follows directly from Proposition \ref{kboundedmapart} and the pigeonhole principle; giving a constant $C(m,k)<mA(m,k)+1$ bounding how many times asynchronously disjoint geodesics can be $m$-close.
\end{proof}

\section{Hyperbolic $k$-geodetic groups are virtually-free}

We will now focus on hyperbolic groups with $k$-geodetic Cayley graphs, with the key result being that they are virtually-free.

We will use a characterisation of virtually-free groups as seen in \cite{GilmanHermillerHoltRees}. Let $e>0$, then a language $L$ over an alphabet $S$ is \emph{$e$-locally excluding} over $S$ if there exists a finite set $F$ of words of length at most $e$ such that any word not in $L$ has a factor in $F$. Then $G$ is virtually-free if and only if there exists a finite inverse-closed generating set $S$ such that the language of geodesics is $e$-locally excluding over $S$ for some $e>0$.

\begin{proposition}\label{prop:VFPart}
Let $k$ be a positive integer. Any hyperbolic $k$-geodetic group $G$ is virtually-free.
\end{proposition}

\begin{proof}
Let $G$ be a hyperbolic group which admits a finite inverse-closed generating set $S$ such that $X = \text{Cay}(G,S)$ is $k$-geodetic. By Proposition \ref{fellowtravellerproperty}, there exists $m > 0$ such that any two geodesics $\gamma_i: [0, n_i] \to X$ with $i \in \{1, 2\}$, $\gamma_1(0) = \gamma_2(0)$ and $d_X(\gamma_1(n_1), \gamma_2(n_2)) \leq 1$, will $m$-fellow travel. We claim that the language of all geodesic words for $G$ with respect to $S$ is an $(C(\lceil m \rceil,k)+1)$-locally excluding language, where $C(\lceil m \rceil,k)$ is the bound given in Corollary \ref{kboundedmclose}. Define the finite set \[F:=\{w\in S^\ast \mid |w|\leq C(\lceil m \rceil,k)+1 \ \text{and} \ w \ \text{not a geodesic}\}.\]
Suppose $w\in S^\ast$ is not a geodesic. Then there exists $u,v\in S^\ast$ and $x\in S$ such that $w=uxv$ and $u$ is a geodesic but $ux$ is not. If the last letter of $u$ is $x^{-1}$, then the factor $x^{-1}x\in F$. Now assume that the last letter of $u$ is not $x^{-1}$, so that the terminal vertex of $ux$ does not lie on the path $u$. Let $w'$ be a geodesic representative of $ux$. Clearly, $|w'|$ is either $|u|$ or $|u|-1$. For compatibility with Definition \ref{allthedefinitions}, we let $w''$ equal $w'$ if $|w'| = |u|$ and $w'x^{-1}$ if $|w'|=|u|-1$. Then $w''$ and $u$ are asynchronously disjoint and $m$-fellow travel. Furthermore, since the terminal vertex of $w'$ does not lie on the path of $u$, there exists words $u_1$ and $u_2$ such that $u=u_1u_2$, $|u_2|>0$ and the words $u$ and $w''$ do not co-travel after $|u_1|$ steps. Since $u$ and $w''$ are $m$-fellow travelling, they must be $\lceil m \rceil$-close after $|u_1|$ steps. By Corollary \ref{kboundedmclose}, $u$ and $w''$ are $\lceil m \rceil$-close at most $C(\lceil m \rceil,k)$ times, so $|u_2|\leq C(\lceil m \rceil,k)$. Therefore, the factor of $w$ given by $u_2x$ is not a geodesic and it appears in $F$. Thus the language of geodesics of $G$ is $(C(\lceil m \rceil,k)+1)$-locally excluding over $S$.
\end{proof}

We also have the following fact regarding non-degenerate triangles and bigons that are useful in later arguments:

\begin{lemma}\label{boundonbigonsandtriangles}
Let $k$ be a positive integer and $G$ a hyperbolic group with inverse-closed generating set $S$ such that $\text{Cay}(G,S)$ is $k$-geodetic. Then the non-degenerate geodesic triangles and bigons in $\text{Cay}(G,S)$ have bounded side-length.
\end{lemma}

\begin{proof}
Let $k$ be a positive integer and suppose that $G$ is $k$-geodetic.  Since $G$ is a hyperbolic group, there exists a $\delta>0$ such that geodesic triangles in $\text{Cay}(G,S)$ are $\delta$-thin. Suppose we have a non-degenerate geodesic triangle in $\text{Cay}(G,S)$ with at least one side of length greater than $2C(\delta,k)$, where $C(\delta,k)$ is found in the proof of Corollary \ref{kboundedmclose}. Then we have asynchronously disjoint geodesics that are $\delta$-close and more than $C(\delta,k)$ times. This contradicts Corollary \ref{kboundedmclose}. This also shows that non-degenerate geodesic bigons have bounded side-length since any non-degenerate geodesic bigon forms a non-degenerate geodesic triangle.
\end{proof}

\section{Centralisers of infinite order elements}

In this section we investigate centralisers of infinite order elements in groups with $k$-geodetic Cayley graph. This will lead to a proof of the second part of Theorem \ref{TheTheorem}, restricting which virtually-free groups can be $k$-geodetic. Our result and proof is motivated by Madlener-Otto's \cite{MadlenerOtto} analogous result for groups presented by finite Church-Rosser Thue systems.

We recall a classical combinatorial result for words over any alphabet. The result is due to Lyndon and Schützenberger and can be found in \cite[Thm. 6.5]{DiekertTextbook}.

\begin{lemma}\label{classicwordresult}
Let $x,y,z\in S^\ast$ be words over an alphabet $S$.
\begin{itemize}
    \item[(a)] If $x\neq \lambda$ and $zx=yz$, then there are $s,t\in S^\ast$ and $q\in\mathbb{N}$ such that $x=st, y=ts$ and $z=(ts)^qt$.
    \item[(b)] If $xy=yx$, then both $x$ and $y$ are powers of the same word.
\end{itemize}
\end{lemma}

\begin{lemma}\label{murraysidea}
Let $k$ be a positive integer, let $G$ be a group with a finite inverse-closed generating set $S$ such that $\text{Cay}(G,S)$ is $k$-geodetic.  If $u\in S^+$ is a primitive word such that $u^r$ is a geodesic for all $r\geq 1$ and $u$ evaluates to $g\in G$, then $C_G(g)=\langle g \rangle$. 
\end{lemma}

\begin{proof}
For the sake of contradiction, suppose that $C_G(g) \neq \langle g \rangle$. Then there exists $h \in C_G(g)$ such that $h \notin \langle g \rangle$.  Let $v\in S^\ast$ be a geodesic word evaluating to $h \in G$. Let $\alpha$ be the ray in $\text{Cay}(G, S)$ from the vertex $1_G$ with label $u^\infty$ and let $\beta$ be the ray from the vertex $h$ with label $u^\infty$.  Then $\alpha$ is the top path, and $\beta$ the bottom path, in a structure shown schematically in Fig. \ref{fig:lls_fromproof}.

\begin{figure}[ht]
    \centering
\begin{tikzpicture}
  \draw (1, 7) to (9, 7);
  \draw (1, 5.5) to (9, 5.5);
    \draw (10, 7) to (12, 7);
  \draw (10, 5.5) to (12, 5.5);
   \draw (9, 7)[dashed] to (10, 7);
  \draw (9, 5.5)[dashed] to (10, 5.5);
  \node at (1,7)[circle,fill,inner sep=1pt]{};
  \node at (1,5.5)[circle,fill,inner sep=1pt]{};
  \node at (3,7)[circle,fill,inner sep=1pt]{};
  \node at (3,5.5)[circle,fill,inner sep=1pt]{};
  \node at (5,7)[circle,fill,inner sep=1pt]{};
  \node at (5,5.5)[circle,fill,inner sep=1pt]{};
  \node at (7,7)[circle,fill,inner sep=1pt]{};
  \node at (7,5.5)[circle,fill,inner sep=1pt]{};
  \node at (9,7)[circle,fill,inner sep=1pt]{};
  \node at (9,5.5)[circle,fill,inner sep=1pt]{};
  \node at (10,7)[circle,fill,inner sep=1pt]{};
  \node at (10,5.5)[circle,fill,inner sep=1pt]{};
  \node at (12,7)[circle,fill,inner sep=1pt]{};
  \node at (12,5.5)[circle,fill,inner sep=1pt]{};
  \draw (1,5.5) to (1,7);
  \draw (3,5.5) to (3,7);
  \draw (5,5.5) to (5,7);
  \draw (7,5.5) to (7,7);
  \draw (9,5.5) to (9,7);
  \draw (10,5.5) to (10,7);
\draw (12,5.5) to (12,7);
\node at (1,7.3) {${1_G}$};
\node at (2,7.3) {${u}$};
\node at (4,7.3) {${u}$};
\node at (6,7.3) {${u}$};
\node at (8,7.3) {${u}$};
\node at (11,7.3) {${u}$};
\node at (12,7.3) {${u}^r$};
\node at (1,5.2) {${h}$};
\node at (2,5.2) {${u}$};
\node at (4,5.2) {${u}$};
\node at (6,5.2) {$u$};
\node at (8,5.2) {$u$};
\node at (11,5.2) {$u$};
\node at (12,5.2) {$u^rh$};
\node at (0.7,6.25) {$v$};
\node at (2.7,6.25) {$v$};
\node at (4.7,6.25) {$v$};
\node at (6.7,6.25) {$v$};
\node at (8.7,6.25) {$v$};
\node at (9.7,6.25) {$v$};
\node at (11.7,6.25) {$v$};
\end{tikzpicture}
    \caption{\small A schematic of $h$ commuting with powers of $u$}
    \label{fig:lls_fromproof}
\end{figure}
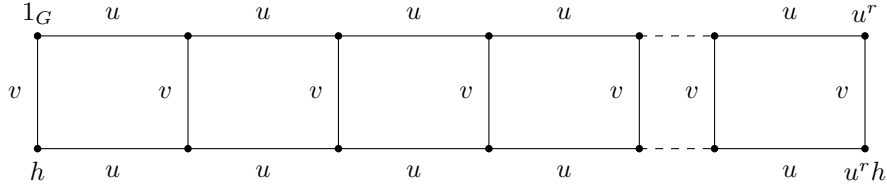

 In this paragraph we show that $\alpha,\beta$ must co-travel but not synchronously, that is, must join after some prefixes $\alpha'\neq\beta'$ of $u^\infty$ as depicted in Fig. \ref{fig:uandudashandudashdash}. Since $\alpha$ and $\beta$ are labelled by the same word but start at distinct vertices in a Cayley graph, they cannot synchronously co-travel. Furthermore, by Proposition \ref{kboundedmapart} they cannot be asynchronously disjoint for arbitrarily large $r$, so we know they must join asynchronously. We then have the diagram depicted in Fig. \ref{fig:uandudashandudashdash}, where $\alpha'$ and $\beta'$ are prefixes of some powers of $u$.

\begin{figure}[ht]
    \centering
    \begin{tikzpicture}
  \draw (1, 7) to (3, 6.25);
  \draw (1, 5.5) to (3, 6.25);
  \draw (3,6.25)[dashed] to (3.5,6.25);
  \node at (1,7)[circle,fill,inner sep=1pt]{};
  \node at (1,5.5)[circle,fill,inner sep=1pt]{};
  \node at (3,6.25)[circle,fill,inner sep=1pt]{};
  \draw (1,5.5) to (1,7);
\node at (1,7.3) {${1_G}$};
\node at (1,5.2) {${h}$};
\node at (2,6.85) {${\alpha'}$};
\node at (2,5.55) {${\beta'}$};
\node at (0.7,6.25) {$v$};
    \end{tikzpicture}
    \caption{\small A depiction of the asynchronous joining}
    \label{fig:uandudashandudashdash}
\end{figure}
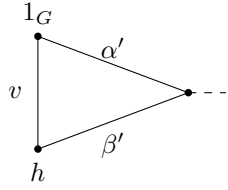

Let $u=u_1\dots u_{|u|}$, so $\alpha'=u^{m_1}u_1\dots u_i$ and $\beta'=u^{m_{2}}u_1\dots u_j$ for some $i,j \in[0,|u|]$. If $i=j$, then $v = u^{m_2 - m_1}$ which is not possible because $h \not \in \langle g \rangle$; so we may assume that $i\neq j$. Now continue moving along $\alpha$ and $\beta$ in Fig. \ref{fig:uandudashandudashdash} starting with $u_{i+1} \dots u_{|u|}$ and $u_{j+1} \dots u_{|u|}$ then powers of $u$. If $\alpha$ and $\beta$ bifurcate, then by applying Proposition \ref{kboundedmapart} starting from the bifurcation point, $\alpha$ and $\beta$ will only remain disjoint for a bounded number of steps. Furthermore, $\alpha$ and $\beta$ cannot bifurcate and meet again more than $\log_2{k}$ times. Hence, $\alpha$ and $\beta$ co-travel forever after some point. First assume $i>j$. Now consider Figure \ref{fig:co-travellingwords}:
\begin{figure}[h]
\centering
\scalebox{0.83}{
$\begin{array}{ll}
u_{i+1}\dots u_{|u|} 
&
  \overbracket{u_1 \ \ \ \ \ \ \ \ \  \dots u_{i-j} \quad u_{i-j+1}\dots u_{|u|}  \ \ \ \ \ \ }^{u} \quad \overbracket{u_1 \ \ \ \ \ \ \ \ \dots u_{i-j} \quad u_{i-j+1}\dots u_{|u|} \ \ \ \ \ }^{u}\\
  u_{j+1}\dots u_{|u|-i+j} 
&
  \underbracket{u_{|u|-i+j+1} \dots u_{|u|}}_x \quad   \underbracket{u_1  \ \ \ \ \ \dots u_{|u|-i+j} \quad  u_{|u|-i+j+1} \dots u_{|u|}}_u
  \quad  \underbracket{u_1\ \ \  \ \dots u_{|u|-i+j}}_y
  \end{array}$
}
    \caption{\small Equating $\alpha$ and $\beta$ as they co-travel forever}
    \label{fig:co-travellingwords}
\end{figure}

By equating words, we deduce that $u^2=xuy$ where  $x=u_{1}\dots u_{|i-j|}$ and $y=u_{1} \dots u_{|u|-i+j}$.  Hence $u=xu'=u''y$ for some words $u'$ and $u''$. Since $|x|+|y|=|u|$, we must have $|u'|=|y|, |u''|=|x|$ so $u=xy$. Hence $xyxy=u^2=xuy=xxyy$, so $xy=yx$. If $i<j$, a similar argument shows that gives $x'y'x'y'=u^2=x'x'y'y'$ for some $x'$ and $y'$. By part (b) of Lemma \ref{classicwordresult}, we find that $u$ is not primitive.
\end{proof}

We will now consider the language of geodesic words for all powers of an infinite order element in a hyperbolic group with $k$-geodetic Cayley graph.

\begin{proposition}\label{powersareregular}
Let $k$ be a positive integer, let $G$ be a group with a finite inverse-closed generating set $S$ such that $\text{Cay}(G,S)$ is $k$-geodetic, and let $g\in G$ be an element of infinite order. For all $n\geq 0$, let $L_n$ be the set of geodesic words for $g^n$ with respect to $S$. If $G$ is hyperbolic then $\bigcup_{n\geq 0}L_n$ is a regular language.
\end{proposition}

\begin{proof}
For a fixed $r>0$, there exists a $p\in\mathbb{N}$ such that $g^{n}\in \{x\in G\mid |x|_{G,S}\geq r\}$ for all $n\geq p$. This is because there is a maximal number of times that powers of $g$ can visit $B_{r}(1_G)$. Diagrammatically we represent the words in $L_n$ as a shaded region from $1_G$ to $g^n$, as seen in Fig. \ref{geodesiccloud}.

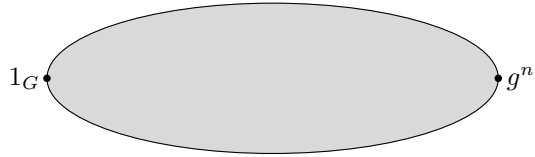
\begin{figure}[ht]
\centering
\begin{tikzpicture}
\filldraw[gray, opacity = .3] (0,0) ellipse (3cm and 1cm);
\draw (0,0) ellipse (3cm and 1cm);
\node at (-3,0)[circle,fill,inner sep=1pt]{};
\node at (3,0)[circle,fill,inner sep=1pt]{};
    \node at (-3.3,0) {$1_G$}; 
    \node at (3.3,0) {$g^n$}; 
\end{tikzpicture}
    \caption{\small Diagrammatic representation of $L_n$}
    \label{geodesiccloud}
\end{figure}

\noindent {\underline{Claim 1}.}
    For a fixed $M_1>0$, there exists an $n_0$ such that \[L_{n}=\{a_{(i)}bc_{(j)} \mid i\in[1,k_1], j\in [1,k_2], k_1k_2\leq k, |b|\geq M_1\},\] for all $n\geq n_0$, where each $a_{(i)}$ is a geodesic representative for some $g_a\in G$ and each $c_{(j)}$ is a geodesic representatives of some $g_c \in G$.
\vspace{1em}

\noindent{\it Proof of Claim 1.} By Lemma \ref{boundonbigonsandtriangles}, there is a bound $B_1$ on the length of non-degenerate geodesic bigons. Furthermore, $B_2=\log_2(k)$ is the maximal number of times that geodesics for the same group element can furcate then rejoin forming non-degenerate geodesic bigons. Then for any $n\geq 0$, $B=B_1B_2$ is the maximum number of total steps that the geodesic words of $g^n$ are not synchronously co-travelling. There is an $n_0\geq 0$ such that $|g^{n}|_{G,S}>B+(M_1-1)(B_2+1)+1$ for all $n\geq n_0$. Hence, the geodesics of $g^n$ all co-travel for at least $(M_1-1)(B_2+1)+1$ steps, and there are at most $B_2+1$ disjoint segments that are separated by a shaded region of non-unique geodesic segments. By the pigeonhole principle at least one of these disjoint segments has length $M_1$. Let the word for such a segment be denoted by $b$. The set of geodesic words from $1_G$ to where the segment $b$ begins are denoted $a_{(i)}$, where $i\in [1,k_1]$ for some $k_1\leq k$, and the geodesic words from where $b$ ends are denoted $c_{(j)}$, where $j\in [1,k_2]$ for some $k_2\leq k$. Note that $k_1k_2\leq k$, since otherwise we would have $|L_n|>k$, contradicting $\text{Cay}(G,S)$ being $k$-geodetic. \hfill $\blacksquare$

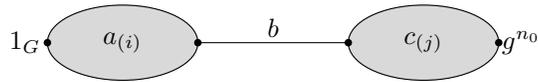
\begin{figure}[h!]
\centering
\begin{tikzpicture}
\filldraw[gray, opacity = .3] (-2,0) ellipse (1cm and 0.5cm);
\draw (-2,0) ellipse (1cm and 0.5cm);
\filldraw[gray, opacity = .3] (2,0) ellipse (1cm and 0.5cm);
\draw (2,0) ellipse (1cm and 0.5cm);
\draw (-1,0) to (1,0);
\node at (-3,0)[circle,fill,inner sep=1pt]{};
\node at (3,0)[circle,fill,inner sep=1pt]{};
\node at (-1,0)[circle,fill,inner sep=1pt]{};
\node at (1,0)[circle,fill,inner sep=1pt]{};
    \node at (-3.3,0) {$1_G$}; 
    \node at (3.3,0) {$g^{n_0}$}; 
\node at (0,0.2) {$b$}; 
\node at (-2,0) {$a_{(i)}$}; 
\node at (2,0) {$c_{(j)}$}; 
\end{tikzpicture}
\caption{\small A unique factor $b$ in all elements of $L_{n}$ with $n\geq n_0$}
\label{alongstretch}
\end{figure}

\noindent {\underline{Claim 2}.}
     For a fixed $M_2>0$, there exists an $n_*$ so that $|g^{n_*+1}|_{G,S}>|g^{n_*}|_{G,S}$ and \[L_{n_*}=\{\alpha_{(i)}\beta\gamma_{(j)} \mid i\in[1,k_1], j\in [1,k_2], |\alpha_{(i)}|,|\gamma_{(j)}|\geq M_2, k_1k_2\leq k\},\] where each $\alpha_{(i)}$ is a geodesic representative for some $g_\alpha\in G$ and each $\gamma_{(j)}$ is a geodesic representative for some $g_\gamma\in G$.
     
\vspace{1em}

\noindent{\it Proof of Claim 2.} 
Take $M_1=2M_2$ from Claim 1. Then shift the prefix of $b$ of length $M_2$ into the left shaded region and shift the suffix of $b$ of length $M_2$ into the right shaded region. Then for each $n\geq n_0$, we have \[L_n=\{\alpha_{(i)}\beta\gamma_{(j)} \mid i\in[1,k_1], j\in [1,k_2], |\alpha_{(i)}|,|\gamma_{(j)}|\geq M_2, k_1k_2\leq k\},\] where each $\alpha_{(i)}$ is a geodesic representative for some $g_\alpha\in G$ and each $\gamma_{(j)}$ is a geodesic representative for some $g_\gamma\in G$. By the opening statement in the proof of this proposition we can choose $n_*\geq n_0$ to be such that $|g^{n_*+1}|_{G,S}>|g^{n_*}|_{G,S}$. \hfill $\blacksquare$

\begin{figure}[h]
    \centering
\begin{tikzpicture}
\filldraw[gray, opacity = .3] (-2,0) ellipse (1cm and 0.5cm);
\draw (-2,0) ellipse (1cm and 0.5cm);
\filldraw[gray, opacity = .3] (2,0) ellipse (1cm and 0.5cm);
\draw (2,0) ellipse (1cm and 0.5cm);
\draw (-1,0) to (1,0);
\node at (-3,0)[circle,fill,inner sep=1pt]{};
\node at (3,0)[circle,fill,inner sep=1pt]{};
\node at (-1,0)[circle,fill,inner sep=1pt]{};
\node at (1,0)[circle,fill,inner sep=1pt]{};
    \node at (-3.3,0) {$1_G$}; 
    \node at (3.3,0) {$g^{n_*}$}; 
\node at (0,0.2) {$\beta$}; 
\node at (-2,0) {$\alpha_{(i)}$}; 
\node at (2,0) {$\gamma_{(j)}$}; 
\end{tikzpicture}
\caption{\small Schematic of $L_{n_*}$}
\label{moreinclouds}
\end{figure}
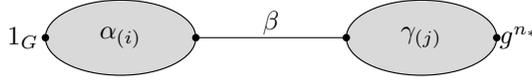

We wish to ensure that the length of geodesics in both shaded regions in Figure \ref{moreinclouds} is at least the maximal side-length of a non-degenerate geodesic triangle (Lemma \ref{boundonbigonsandtriangles}), which we denote by $\Delta$. Hence, in Claim 2 choose $n_*$ to correspond to some $M_2\geq \Delta$. Now consider the geodesics from $1_G$ to $g^{n_*+1}$ depicted in Fig. \ref{1tognplus1}.
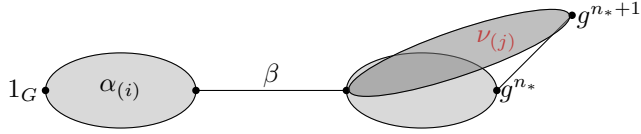
\begin{figure}[h]
    \centering
\begin{tikzpicture}
\filldraw[gray, opacity = .3] (-2,0) ellipse (1cm and 0.5cm);
\draw (-2,0) ellipse (1cm and 0.5cm);
\filldraw[gray, opacity = .3] (2,0) ellipse (1cm and 0.5cm);
\draw (2,0) ellipse (1cm and 0.5cm);
\draw (-1,0) to (1,0);
\draw (4,1) to (3,0);
\node at (-3,0)[circle,fill,inner sep=1pt]{};
\node at (3,0)[circle,fill,inner sep=1pt]{};
\node at (-1,0)[circle,fill,inner sep=1pt]{};
\node at (1,0)[circle,fill,inner sep=1pt]{};
\node at (4,1)[circle,fill,inner sep=1pt]{};
    \node at (-3.3,0) {$1_G$}; 
    \node at (3.3,0) {$g^{n_*}$}; 
    \node at (4.5,1) {$g^{n_*+1}$}; 
\node at (0,0.2) {$\beta$}; 
\node at (-2,0) {$\alpha_{(i)}$}; 
\node at (3,0.66) {\textcolor{red}{$\nu_{(j)}$}};
\filldraw[gray, opacity = .5][rotate around={19:(2.5,0.5)}] (2.5,0.5) ellipse (1.55cm and 0.3cm);
\draw[rotate around={19:(2.5,0.5)}] (2.5,0.5) ellipse (1.55cm and 0.3cm);
\end{tikzpicture}
    \caption{\small Geodesics from $1_G$ to $g^{n_*+1}$}
    \label{1tognplus1}
\end{figure}

Since $|\gamma_{(i)}|\geq\Delta$, the geodesics in $L_{n_*+1}$ share a prefix up to the end of the word $\beta$ to an element of $L_{n_*}$. Hence, $L_{n_*+1}=\{\alpha_{(i)}\beta\nu_{(j)}\mid i\in [1,k_1], j\in [1,k_3] \}$ for some positive integer $k_3$ with $k_1k_3\leq k$. Instead, let us now consider the geodesics from $g^{-1}$ to $g^{n_*}$ depicted in Fig. \ref{ginvtogn}:

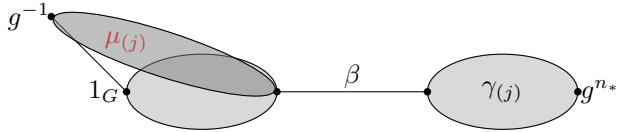
\begin{figure}[h]
    \centering
\begin{tikzpicture}
\filldraw[gray, opacity = .3] (-2,0) ellipse (1cm and 0.5cm);
\draw (-2,0) ellipse (1cm and 0.5cm);
\filldraw[gray, opacity = .3] (2,0) ellipse (1cm and 0.5cm);
\draw (2,0) ellipse (1cm and 0.5cm);
\draw (-1,0) to (1,0);
\draw (-4,1) to (-3,0);
\node at (-3,0)[circle,fill,inner sep=1pt]{};
\node at (3,0)[circle,fill,inner sep=1pt]{};
\node at (-1,0)[circle,fill,inner sep=1pt]{};
\node at (1,0)[circle,fill,inner sep=1pt]{};
\node at (-4,1)[circle,fill,inner sep=1pt]{};
    \node at (-3.3,0) {$1_G$}; 
    \node at (3.3,0) {$g^{n_*}$}; 
    \node at (-4.3,1) {$g^{-1}$}; 
\node at (0,0.2) {$\beta$}; 
\node at (2,0) {$\gamma_{(j)}$}; 
\node at (-3,0.66) {\textcolor{red}{$\mu_{(j)}$}};
\filldraw[gray, opacity = .5][rotate around={162:(-2.5,0.5)}] (-2.5,0.5) ellipse (1.55cm and 0.3cm);
\draw[rotate around={162:(-2.5,0.5)}] (-2.5,0.5) ellipse (1.55cm and 0.3cm);
\end{tikzpicture}
    \caption{\small Geodesics from $g^{-1}$ to $g^{n_*}$}
    \label{ginvtogn}
    \end{figure}
    
Since $|\alpha_{(i)}|\geq\Delta$, the geodesics in $L_{n_*+1}$ share a suffix to an element of $L_{n_*}$ up to the start of the word $\beta$. Hence, $L_{n_*+1}=\{\mu_{(i)}\beta\gamma_{(j)}\mid i\in[1,k_4], j\in [1,k_2] \}$ for some positive integer $k_4$ with $k_4k_2\leq k$. Therefore, for each $i \in [1,k_1]$ and $j\in [1,k_3]$ there is an $l\in [1,k_2]$ and $ m \in [1,k_4]$ such that $\alpha_{(i)}\beta\nu_{(j)}=\mu_{(m)}\beta\gamma_{(l)}$. Since $|g_{n_*+1}|_{G,S}>|g_{n_*}|_{G,S}$, we know that $\alpha_{(i)}$ is a prefix of $\mu_{(m)}$ and $\gamma_{(i)}$ is a suffix of $\nu_{(m)}$. Then we have $\beta x=y\beta$ where $\nu_{(l)}=\alpha_{(i)}y$ and $\mu_{(m)}=x\gamma_{(j)}$. Invoking part (a) of Lemma \ref{classicwordresult} we have $t,s\in S^\ast$ such that $x=st, y=ts$ and $\beta=(ts)^qt$ for some $q\in \N$. So \[L_{n_*}=\{\alpha_{(i)}(ts)^qt\gamma_{(j)}\mid i\in [1,k_1], j\in [1,k_2] \}\] and \[L_{n_*+1}=\{\alpha_{(i)}(ts)^{q+1}t\gamma_{(j)}\mid i\in [1,k_1], j\in [1,k_2] \}.\] Since the prefixes $\alpha_{(i)}$ and suffixes $\gamma_{(j)}$ are preserved we can inductively deduce that \[L_{n_*+c}=\{\alpha_{(i)}(ts)^{q+c}t\gamma_{(j)}\mid i\in [1,k_1], j\in [1,k_2] \}.\] Hence, we conclude \[\bigcup_{n\geq 0}L_n=(\bigcup_{n\geq 0}^{n_*-1}L_{n})\cup\{\alpha_{(i)}(ts)^{q+c}t\gamma_{(j)}\mid i\in [1,k_1], j\in [1,k_2], c\geq 0\},\] which is regular.
\end{proof}

We are now ready to prove the second part of Theorem \ref{TheTheorem}. The proof follows from the proofs of \cite[Thm. 2.3 \& Corollary 2.4]{MadlenerOtto}, but we include it with our own notation for completeness.

\begin{proposition}\label{prop:centraliserpart}
Let $k$ be a positive integer. The centraliser of any infinite order element is infinite cyclic in a hyperbolic $k$-geodetic group.
\end{proposition}

\begin{proof}
Let $G$ be a hyperbolic $k$-geodetic group, and let $S$ be a finite generating set such that $\text{Cay}(G, S)$ is $k$-geodetic. Let $g\in G$ be an infinite order element.

By Proposition \ref{powersareregular}, if $L_n$ is the set of geodesic words for $g^n$, then $\bigcup_{n\geq 0}L_n$ is regular. By the pumping lemma for regular languages there is a subset of $\bigcup_{n\geq 0}L_n$ given by $\{xw^iz\mid i\geq 0\}$ such that $|w|\neq 0$. Let $y$ be the primitive root of $w$, so $w=y^m$ for some $m\geq 0$. Since $w^i$ is a geodesic for all $i\geq 0$, all powers of $y$ are geodesics. For any $i\geq 0$ there exists an index $j_i$ such that $xy^{mi}z$ is a geodesic representative of $g^{j_i}$. Since there are at most $k$ representatives for a given $g^{j_i}$, we can choose an $n\geq0$ such that $j_n<j_{n+1}$. 

Since $xw^{n+1}z\equiv g^{j_{n+1}}$ and $xw^{n}z\equiv g^{j_{n}}$ we find that $g^{j_{n+1}-j_n}\equiv z^{-1}wz$. Let $h\in C_G(g)$, so \[y^m(zhz^{-1})\equiv zz^{-1}wzhz^{-1}\equiv zg^{j_{n+1}-j_n}hz^{-1}\equiv zhg^{j_{n+1}-j_n}z^{-1}\equiv(zhz^{-1})y^m,\] and by Lemma \ref{murraysidea} $zhz^{-1}\in \langle y\rangle$. We have shown that any $h\in C_G(g)$ is contained in  $\langle z^{-1}yz\rangle\cong \Z$, so $C_G(g)\leq \mathbb{Z}$. The result follows since the only non-trivial subgroups of an infinite cyclic group are infinite cyclic.
\end{proof}

Propositions \ref{prop:VFPart} and \ref{prop:centraliserpart} together yield Theorem \ref{TheTheorem}. Furthermore, Proposition \ref{prop:centraliserpart} immediately yields the following result.

\begin{corollary}\label{finiteordercentraliser}
Let $k$ be a positive integer and let $G$ be a group with a finite inverse-closed generating set $S$ such that $\text{Cay}(G,S)$ is $k$-geodetic.  If $G$ is hyperbolic, and $g,h\in G$ are commuting non-trivial elements, then:
\begin{itemize}
    \item[(a)] If $g$ has finite order, then $gh$ and $h$ have finite order.
    \item[(b)] If $g$ has infinite order, then $h$ has infinite order and either $g$ and $h$ are inverses or $gh$ has infinite order.
\end{itemize}
\end{corollary}

\begin{remark}
In general, the centraliser of an infinite order element of a virtually-free group is virtually-cyclic (see \cite[III. $\Gamma.$ Cor. 3.10 ]{MetricSpacesofNon-PositiveCurvature}), so Proposition \ref{prop:centraliserpart} excludes many virtually-free groups from being $k$-geodetic groups.
\end{remark}

\section*{Acknowledgments}

This research is supported by  Australian Research Council grant DP210100271. The authors would like to acknowledge the referee for helpful comments and corrections.

\bibliographystyle{plain}
\bibliography{kgeodeticbib}

\begin{thebibliography}{10}

\bibitem{MetricSpacesofNon-PositiveCurvature}
Martin~R. Bridson and Andr\'{e} Haefliger.
\newblock {\em Metric spaces of non-positive curvature}, volume 319 of {\em
  Grundlehren der mathematischen Wissenschaften [Fundamental Principles of
  Mathematical Sciences]}.
\newblock Springer-Verlag, Berlin, 1999.

\bibitem{DiekertTextbook}
Volker Diekert, Manfred Kufleitner, Gerhard Rosenberger, and Ulrich Hertrampf.
\newblock {\em Discrete algebraic methods}.
\newblock De Gruyter Textbook. De Gruyter, Berlin, 2016.
\newblock Arithmetic, cryptography, automata and groups.

\bibitem{WordProcessingInGroups}
David B.~A. Epstein, James~W. Cannon, Derek~F. Holt, Silvio V.~F. Levy,
  Michael~S. Paterson, and William~P. Thurston.
\newblock {\em Word processing in groups}.
\newblock Jones and Bartlett Publishers, Boston, MA, 1992.

\bibitem{Frasser}
Carlos~E. Frasser.
\newblock The open problem of finding a general classification of geodetic
  graphs, 2022.

\bibitem{GilmanHermillerHoltRees}
Robert~H. Gilman, Susan Hermiller, Derek~F. Holt, and Sarah Rees.
\newblock A characterisation of virtually free groups.
\newblock {\em Arch. Math. (Basel)}, 89(4):289--295, 2007.

\bibitem{MadlenerOtto}
Klaus Madlener and Friedrich Otto.
\newblock Commutativity in groups presented by finite {C}hurch-{R}osser {T}hue
  systems.
\newblock {\em RAIRO Inform. Th\'{e}or. Appl.}, 22(1):93--111, 1988.

\bibitem{Ore}
Oystein Ore.
\newblock {\em Theory of graphs}.
\newblock Third printing, with corrections. American Mathematical Society
  Colloquium Publications, Vol. XXXVIII. American Mathematical Society,
  Providence, R.I., 1967.

\bibitem{Papasoglu}
Panagiotis Papasoglu.
\newblock {\em Geometric methods in group theory}.
\newblock ProQuest LLC, Ann Arbor, MI, 1993.
\newblock Thesis (Ph.D.)--Columbia University.

\bibitem{ShapiroPascal}
Michael Shapiro.
\newblock Pascal's triangles in abelian and hyperbolic groups.
\newblock {\em J. Austral. Math. Soc. Ser. A}, 63(2):281--288, 1997.

\bibitem{bigeodetic}
N.~Srinivasan, J.~Opatrn\'{y}, and V.~S. Alagar.
\newblock Bigeodetic graphs.
\newblock {\em Graphs Combin.}, 4(4):379--392, 1988.

\end{thebibliography}

\end{document}